\documentclass[12pt]{article}

\usepackage{amsmath,amsfonts,amssymb,amsthm}
\usepackage{graphicx} 
\usepackage{authblk}
\usepackage[letterpaper]{geometry}
\usepackage{nicefrac}
\usepackage{multirow}
\usepackage{float}

\newtheorem{thm}{Theorem}
\newtheorem{lem}[thm]{Lemma}
\newtheorem{cor}[thm]{Corollary}

\title{On the cogrowth of Thompson's group $F$\footnote{
The first author acknowledges support from ARC projects DP110101104 and FT110100178.
 The second author thanks NSERC of Canada for financial support.}}
\author[1]{Murray Elder}
\author[2]{Andrew Rechnitzer}
\author[2]{Thomas Wong}

\affil[1]{School of Mathematical \& Physical Sciences, The~University~of~Newcastle}
\affil[2]{Department of Mathematics, University of British Columbia}

\begin{document}
\maketitle


\abstract{We  investigate the cogrowth  and distribution of geodesics in R. Thompson's group $F$.}

\section{Introduction}

In this article we study the cogrowth  and distribution of geodesics in  Richard Thompson's group $F$, in an attempt to decide experimentally whether or not  $F$ is amenable.

The {\em cogrowth} of a finitely generated group $G$ is defined as follows. Suppose $S=\{a_1,\ldots, a_k\}$ generates $G$\footnote{Formally, we consider $G$ as the epimorphic image from the free monoid generated by 
$S\cup S^{-1}$, rather than $S\cup S^{-1}$ as being a subset of $G$}, 
and 
consider the Cayley graph $\mathcal{G}$ of $(G,S)$. 
 Let $r_n$ be the number of
paths in this graph of length $n$ starting and ending at the identity
element --- let us call such paths {\em returns}. Since we can concatenate any two
such paths to get another we have
\begin{align}
  r_n r_k \leq r_{n+k}
\end{align}
and then by Fekete's lemma (see, for example, \cite{vanLint2001}) 
\begin{align}
  \rho = \limsup_{n \to \infty} r_n^{\nicefrac{1}{n}}
\end{align}
exists. This constant is called the {\em cogrowth} for $(G,S)$. Since we consider generators and their inverses to label distinct edges in $\mathcal G$,
 then  $\rho \leq 2k$.

The connection between this growth rate and amenability was established by Grigorchuk and independently by Cohen:
\begin{thm}[\cite{Grigorchuk1980,Cohen1982}]
\label{thm:amenable}
Let $G,S$ and
$\rho$ be as above.  $G$ is amenable if and only if $\rho= 2k$.
\end{thm}

Let $p_n$ be the number of returns of length $n$ on $\mathcal{G}$ which do not
contain immediate reversals. Again  concatenation shows that $p_n$ is supermultiplicative so Fekete's lemma gives
\begin{align}
  \alpha &= \limsup_{n \to \infty} p_n^{\nicefrac{1}{n}} 
\end{align}
exists. In this case since there are $2k(2k-1)^{n-1}$ freely reduced words of length $n$ in the $2k$ generators and their inverses, we have $\alpha\leq 2k-1$. 

The previous theorem can then be restated as:
\begin{thm}[\cite{Grigorchuk1980,Cohen1982}]
\label{thm:amenable_reduced}
Let $G,S$ and
$\alpha$ be as above.
 $G$ is amenable if and only if $\alpha= 2k-1$.
\end{thm}

Note that $\limsup$s are required since, for example, if $G$ has a presentation where all relators have even length, the number of returns of odd length (with or without immediate reversals) is 0.

In this article we compute bounds on the cogrowth rates of a number of 2-generator groups: Thompson's group $F$, the free and free abelian groups on 2 generators, Baumslag-Solitar groups, and various wreath products. Each of these examples, apart from $F$, is known to be either amenable or non-amenable. We compare the data obtained for $F$ against these examples, to see whether $F$ behaves more like an amenable  or a non-amenable group.

The question of the amenability of Thompson's group  $F$ has captivated many researchers for some time, initially since $F$ has exponential growth but no nonabelian free subgroups, making it a prime candidate for a counterexample to von Neumann's conjecture that a group is non-amenable if and only if it contains a nonabelian free subgroup.  In 1980 Ol'shanskii constructed a finitely generated non-amenable group with no nonabelian free subgroups \cite{Olshan1980}, and in 1982 Adyan gave further examples \cite{Adyan}. In 2002 Ol'shanskii and Sapir constructed finitely presented examples \cite{Sapir2002}. In spite of these results the amenability or non-amenability of $F$ remains an intensely studied problem.

In the second half of the article we extend our techniques to study the distribution of geodesic words in Thompson's group. 

This work is in the same spirit as previous papers by Burillo, Cleary and Wiest  \cite{Burillo2007}, and Arzhantseva, Guba,  Lustig, and Pr\'eaux \cite{AGLP},
who also applied computational techniques to consider the amenability of $F$.
We refer the reader to these papers for more background on Thompson's group 
and the problem of deciding its amenability computationally.

The  article is organised as follows. 
 In Section~2 we compute rigorous lower bounds on the cogrowth by
computing the dominant eigenvalue of the adjacency matrix of truncated Cayley
graphs. We then extrapolate these bounds to estimate the cogrowth and compare
and contrast those extrapolations for $F$ and other groups.
 In Section~3 we use a weighted random sampling of random words in the
generators to estimate the exponential growth rate of trivial words in several
different groups. As a byproduct we estimate the distribution of geodesic
lengths as a function of word-length. 

\section{Bounding returns and cogrowth}

\subsection{Bounding the number of returns}
Consider the Cayley graph $\mathcal{G}$ of some group $G$ with finite generating set --- for the discussion at
hand, let us assume that $G$ is generated by two nontrivial elements $a,b$.

As noted above,  an upper bound for the cogrowth $\rho$ is $4$. 
We can compute lower bounds for the number of returns, and thus the cogrowth, as follows.

Consider the following sequence of  finite connected subgraphs, $\mathcal{G}_N$ of $N$ vertices that
contain the identity. 


Set $\mathcal G_1$ to be the identity vertex. 
Record the list of edges incident to $\mathcal G_1$. Define $\mathcal G_2,\mathcal G_3, \ldots$ by appending edges from this list, one at a time.
Once the list is exhausted (so $\mathcal G_N=B(1)$), repeat the process.
It follows that for each $\mathcal G_N$ there is an $R$ so
that $B(R) \subseteq \mathcal{G}_N \subseteq B(R+1)$. 

We can then define $r_{N,n}$ be the number of returns of length $n$ in $\mathcal{G}_N$. 
Since $\mathcal G_N\subset \mathcal G_{N+1}$, the sequence $\{r_{N,n}\}$ is supermultiplicative, so 
 $\displaystyle \rho_N = \limsup_{n \to \infty} r_{N,n}^{\nicefrac{1}{n}}$ exists by Fekete's lemma. Further
we must have $r_n \geq r_{N,n}$ and so $\rho \geq \rho_{N}$. Hence we can bound
$\rho$ by computing $\rho_N$.

Using the Perron-Frobenius theorem (in one of its many guises ---
Proposition V.7 from \cite{Flajolet2009} for example) the growth rate $\rho_N$
of such paths on $\mathcal{G}_N$ is given by the dominant eigenvalue of the
corresponding adjacency matrix, provided it is irreducible. We construct
$\mathcal{G}_N$ so that it is connected and so the corresponding
adjacency matrix is be irreducible. 

In some cases we can also demonstrate that the adjacency matrix is aperiodic,
which implies that the dominant eigenvalue is simple and dominates all other
eigenvalues. This also implies that the corresponding generating function $\sum
p_{N,n}z^n$ has a simple pole at the reciprocal of that eigenvalue.
Unfortunately many of the matrices we study are not aperiodic, but they do have
period $2$. 

Perhaps the easiest way to prove that the matrix is aperiodic is to
show the existence of two circuits of relatively prime length (see
chapter V.5 in \cite{Flajolet2009} for example). In order to show that a matrix
has period 2 it suffices (providing the matrix is finite) to show that if there is a
path of length $k$ between any two nodes, then there is a path of length
$k+2\ell$ between those two nodes for any $\ell$. 

It follows that the adjacency matrix for $\mathcal G_N$ is aperiodic whenever the group $G$ has a presentation with an odd length relator (since $aa^{-1}$ and the odd length relator form circuits of relatively prime lengths) and if all relators have even length, the matrix has period 2 (since a path of length $k$ can be made into a path of length $k+2\ell$ by inserting $(aa^{-1})^{\ell}$).

In particular we have that 
subgraphs of  Baumslag-Solitar groups  $BS(p,q)$ 
with $p+q$ odd are aperiodic, while subgraphs of $BS(p,q)$ with $p+q$ even, Thompson's group $F$, $\mathbb Z^2$, $\mathbb{Z} \wr
\mathbb{Z}$ and $F_2$ (the free group on 2 generators), all with the usual generating sets, have period $2$.

Since all of above groups except $\mathbb Z^2$
grow exponentially, and $B(R) \subseteq \mathcal{G}_N \subseteq B(R+1)$, then  the radius of $\mathcal{G}_N$ is $O(\log N)$. In the case of $\mathbb Z^2$ the radius of  $\mathcal{G}_N$ is $O(\sqrt N)$


We used this method to compute  $\rho_N$ for a selection of groups. However, we found significantly better
bounds by considering only freely reduced words, {\em i.e.} paths that did not contain immediate reversals, essentially since there is less to count.


\subsection{Bounding the cogrowth}
Let $p_n$ be the number of returns of length $n$ on $\mathcal{G}$ which do not
contain immediate reversals. We similarly define $p_{N,n}$ to be similar paths
on the subgraph $\mathcal{G}_N$. Again we define the exponential growth of
these quantities by
\begin{align*}
  \alpha &= \limsup_{n \to \infty} p_n^{\nicefrac{1}{n}} &
  \alpha_N &= \limsup_{n \to \infty} p_{N,n}^{\nicefrac{1}{n}} 
\end{align*}
and $\alpha \geq \alpha_N$. 

In this case, we cannot now simply concatenate two freely reduced paths to
obtain another freely reduced path since it may create an immediate reversal.
Thus we do not have similar supermultiplicative relations. We can, however,
relate $r_n$ to $p_n$ and $\rho$ to $\alpha$ using the following result of Kouksov
\cite{Kouksov1998} which we have specialised to the case of 2 generator groups.
\begin{lem}(from \cite{Kouksov1998})
\label{lem:RC}
Let $R(z) = \sum r_n z^n$ and $C(z) = \sum p_n z^n$ be the generating functions
of returns and freely reduced returns respectively. Then
\begin{align*}
C(z) &= \frac{1-z^2}{1 + 3 z^2}R\left(\frac{z}{1 + 3 z^2}\right) &\mbox{and
equivalently}\\
R(z) & = \frac{-1+2\sqrt{1 -12z^2}}{1 - 16 z^2}C\left(\frac{1-\sqrt{1 -
12 z^2}}{6 z}\right).
\end{align*}
\end{lem}
A careful generating function argument gives the second equation (and the first
is simply its inverse). Consider any freely reduced returning path; it can be
mapped to an infinite set of returning paths by replacing each edge $s$ by any
returning path in the free group on 2 generators that does start with $s^{-1}$.
At the level of generating functions, this is exactly the substitution $$\displaystyle z
\mapsto \frac{1-\sqrt{1 - 12 z^2}}{6 z}.$$

A very general result for generating functions then links the dominant
singularity of $R(z)$ to the value of $\rho$:
\begin{thm}(\cite{Flajolet2009}, page 240.)
\label{thm:bowtie}
If $f(z)$ is analytic at $0$ and $\rho$ is the modulus of a singularity nearest
to the origin, then the coefficient $f_n = [z^n]f(z)$ satisfies:
\begin{equation*}
  \limsup_{n\to\infty} |f_n|^{-\nicefrac{1}{n}} =  \rho
\end{equation*}
\end{thm}
Combining these two results (and using the positivity of $r_n, p_n$) we obtain
\begin{cor}
\label{cor:transfer}
The constants $\rho$ and $\alpha$ are related by
\begin{align*}
  \rho  = \frac{\alpha^2+3}{\alpha}
\end{align*}
Further if $\beta$ is a lower bound for $\alpha$, then $\displaystyle \beta' = \frac{\beta^2
+3}{\beta}$ is an lower bound for $\rho$.
\end{cor}
We are unable to prove a similar exact relationship between $\rho_N$ and
$\alpha_N$, but we do have the following bound:
\begin{lem}
\label{lem:rhoNalphaN}
For a fixed value of $N$ we have $\displaystyle \rho_N \leq \frac{\alpha_N^2 + 3}{\alpha_N}$.
\end{lem}
\begin{proof}
 Consider the generating functions of returns and freely reduced returns on
$\mathcal{G}_n$.
\begin{align*}
 R_N(z) &= \sum_{n\geq0} r_{N,n}z^n & C_N(z) &= \sum_{n\geq0} p_{N,n}z^n
\end{align*}
It suffices to show that 
\begin{align*}
C_N(z) & \geq \frac{1-z^2}{1 + 3 z^2}R_N\left(\frac{z}{1 + 3 z^2}\right) &
\mbox{or equivalently }\\
R_N(z) & \leq \frac{-1+2\sqrt{1 -12z^2}}{1 - 16 z^2}C_N\left(\frac{1-\sqrt{1 -
12 z^2}}{6 z}\right)
\end{align*}
within the respective radii of convergence. Let $\omega$ be any freely reduced
returning path of length $k$ in $\mathcal{G}_N$ --- this path contributes $z^k$
to the generating function $C_N(z)$. The substitution $\displaystyle z \mapsto \frac{1-\sqrt{1
- 12 z^2}}{6 z}$ maps $\omega$ to an infinite set of non-reduced returning paths
by replacing each edge with freely reduced words from $F_2$. Some of the
resulting words will lie entirely within $\mathcal{G}_N$ and so be enumerated by
the generating function $R_N(z)$. However an infinite number of these words
will not be contained in $\mathcal{G}_N$. These words are enumerated by
$$ \frac{-1+2\sqrt{1 -12z^2}}{1 - 16 z^2}C_N\left(\frac{1-\sqrt{1 -
12 z^2}}{6 z}\right)$$ but not by $R_N(z).$ Thus the inequality follows.
\end{proof}

To compute $\alpha_N$ we relate it to the dominant eigenvalue of an adjacency
matrix. Unfortunately there is no simple way to reuse the adjacency matrix of
$\mathcal{G}_N$, in order to enumerate paths without immediate reversal. Instead
we construct a new graph $\mathcal H_N$ which encodes freely reduced paths in $\mathcal G$ as follows:
$\mathcal H_N$ has $N$ vertices labeled by pairs $(1,-)$ or $(g,s)$ where $g\in G$ and $s\in S$. The vertex $(1,-)$ corresponds to being at the identity vertex of $\mathcal G$, and $(g,s)$  to being at the group element $g\in\mathcal G$ after having just read a letter $s$.
The edges of $\mathcal{H}_N$ are
\begin{align}
  E(\mathcal{H}_N) &= \{ ((g,s),(h,t)) \in ( V(\mathcal{H}_N) )^2 \;|\;
 h=gt \mbox{ and } st \neq 1  \}
\end{align}
So a path $(1,-), (g_1,s_1),(g_2,s_2),\dots (g_k,s_k)$ corresponds to a path in $\mathcal G$ starting at $1$ with $g_1=s_1,g_2=s_1s_2,\ldots g_k=s_1s_2\ldots s_k$ a freely reduced word.




 We construct $\mathcal{H}_N$ using a breadth-first
search similar to the construction of $\mathcal G_N$, starting with $\mathcal H_1=(1,-)$ and appending vertices one at a time so that $\{g\in G \ | \ (g,s)\in \mathcal H_N\}$ lies between two balls of a given radius.
It follows that  $\mathcal{H}_N$  is necessarily connected,  and the corresponding adjacency
matrices are irreducible.


We then compute the growth rate of paths (and so freely reduced returns) on
$\mathcal{H}_N$ by computing the dominant eigenvalue of the corresponding
adjacency matrix.

\subsection{Exact lower bounds}

Since each node of $\mathcal{G}_N$ has outdegree at most $4$ and those 
of the $\mathcal{H}_N$ excluding $(1,-)$ have outdegree at most $3$ (vertices on the boundary may
have smaller degree) the corresponding adjacency matrices are sparse. We found
that the power method and Rayleigh quotients (see \cite{Saad2003} for example)
converged very quickly to the dominant eigenvalue and so the growth rate.

We constructed $\mathcal{G}_N$ and $\mathcal{H}_N$ for many different values
of $N$ ranging between $10^2$ and $10^7$. Our calculations on Thompson's group
$F$ as well as the Baumslag-Solitar groups $BS(2,2), BS(2,3)$ and $BS(3,5)$,
yielded the following result.
\begin{thm}
\label{thm:rigorous_bounds}
The following are exact lower bounds on the cogrowth, $\alpha$, of the indicated
groups.
 \begin{align*}
    BS(2,2) \geq 2.5904 && BS(2,3) \geq 2.42579 && BS(3,5) \geq 2.06357 \\
    \mbox{Thompson's group} \geq 2.17329
 \end{align*}
This implies that the growth rate of all trivial words, $\rho$, in these groups
are bounded as indicated.
 \begin{align*}
     BS(2,2) \geq 3.78522 && BS(2,3) \geq 3.66250 && BS(3,5) \geq 3.51736 \\
    \mbox{Thompson's group} \geq 3.55368
 \end{align*}
\end{thm}
Note that all of these bounds were computed using information from
$\mathcal{H}_N$ and Corollary~\ref{cor:transfer}. We observed that the bounds
obtained from $\mathcal{G}_N$ were worse --- typically differing in the second
or third significant digit. We also note that the above result for $BS(2,3)$ is
improves on a result in   \cite{Dykema2010} (the preprint was withdrawn by the authors since it contained an error).


These computations were done on a desktop computer using about $4$Gb of
memory. It should be noted that while our techniques require both exponential
time and memory, it was memory that was the constraining factor. We did
implement some simple space-saving methods. Perhaps the most effective of these
was to store elements as geodesic words in the generators rather than as their
more standard normal forms (eg tree-pairs for Thompson's group or words in the
normal form implied by Britton's lemma for the Baumslag-Solitar groups). These
geodesic words could then be stored as bit-strings rather than ASCII strings. We
believe that by running these computations on a computer with more memory we
could improve the bounds, but the returns are certainly diminishing.

\subsection{Extrapolation and comparison}
The results of the previous section can be extended by considering the
sequence of lower bounds $\alpha_N$ and using simple numerical
methods to extrapolate them to $N \to \infty$. This required only minimal
changes to our computations; after computing the adjacency matrix of
$\mathcal{H}_N$ for the maximal value of $N$, we computed the dominant
eigenvalue of submatrices. The corresponding estimate of the eigenvector was
then used as an initial vector for estimating the eigenvalue of the next
submatrix. This meant that we could compute a sequence of lower bounds in not
much more time than it took to compute our best bounds.

In Figure~\ref{fig ev nonamen} we have plotted $\alpha_N$ against
$\nicefrac{1}{\log N}$ for three non-amenable  Baumslag-Solitar groups and $F$.
\begin{figure}[ht!]
 \begin{center}
    \includegraphics[height=8cm]{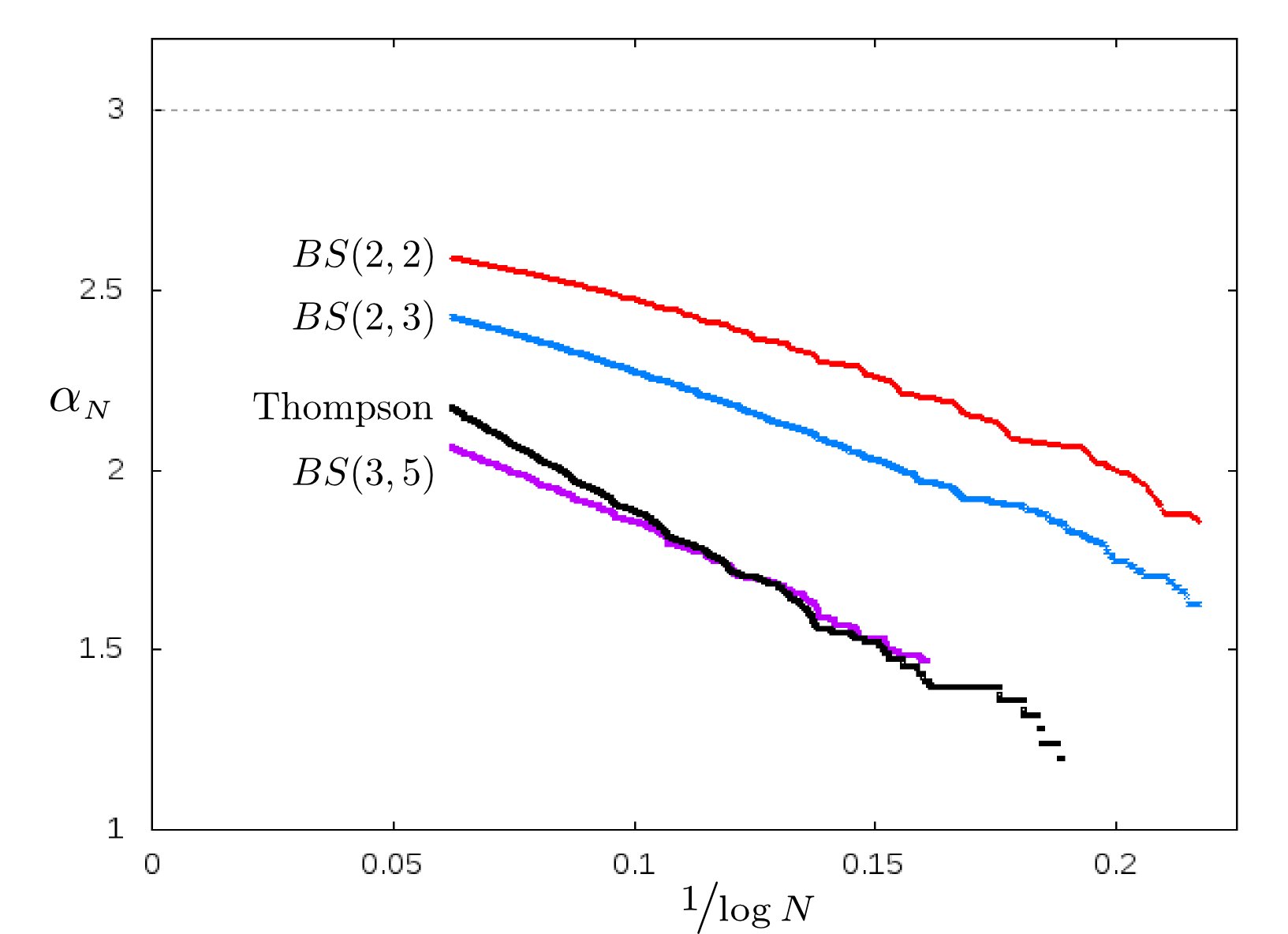}
 \end{center}
  \caption{A plot of cogrowth lower bounds $\alpha_N$ against
$\nicefrac{1}{\log N}$. We see that the groups that known to be non-amenable
are converging to numbers strictly below $3$. The Thompson's group sequence has
a clear upward inflection (as $N \to \infty$ or $\nicefrac{1}{\log N} \to 0$)
and so it is difficult to estimate whether the limit is $3$ or less than $3$.}
 \label{fig ev nonamen}
\end{figure}
 We found that this gave approximately linear plots and
so this suggests that 
\begin{align*}
  \alpha_N \approx \alpha_\infty - \nicefrac{\lambda}{\log N}.
\end{align*}
Since $\alpha_N$ is a monotonically increasing sequence 
and is bounded above by $3$, we have that $\alpha_N \to \alpha_\infty$
exists. Unfortunately we cannot prove that $\alpha_\infty = \alpha$, but
certainly $\alpha_\infty \leq \alpha$.

Note that the curves terminate at $N=10^7$ but start at different $N$-values.
This is because the graphs $\mathcal{G}_N$ do not contain freely-reduced loops
for small values of $N$. The smallest value of $N$ for which
$\mathcal{G}_N$ contains a freely reduced loop depends on the length of the
relations of the group and on the details of the breadth-first search used to
construct the graph. 

One can observe that the Baumslag-Solitar groups all seem to behave similarly
and that the sequences of bounds are clearly converging to constants strictly
less than $3$. This is completely consistent with the non-amenability of these
groups. Thompson's group behaves quite differently --- in particular we see that
the curve has some upward inflection (as $x \to 0$) and it makes it very unclear
as to whether or not $\alpha_\infty$ converges to 3 or below 3.

For the sake of comparison we decided to repeat the above analysis for a set of
amenable groups and so we computed similar sequences of lower bounds for
$BS(1,2), BS(1,3), \mathbb{Z}^2$ and $\mathbb{Z}\wr\mathbb{Z}$. These results
are plotted in Figure~\ref{fig ev amen}.
\begin{figure}[ht!]
 \begin{center}
  \includegraphics[height=8cm]{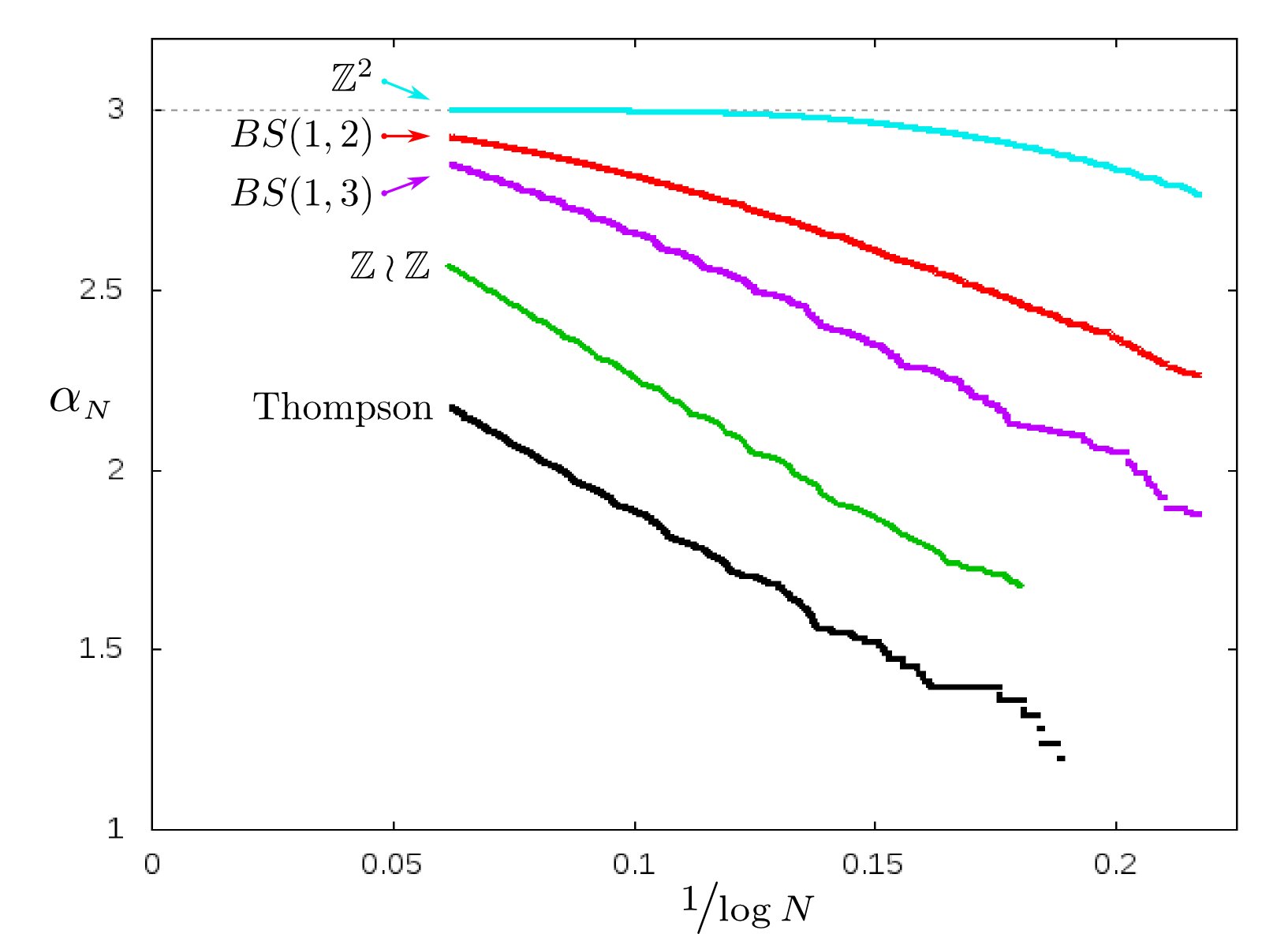}
 \end{center}
  \caption{A plot of cogrowth lower bounds $\alpha_N$ against $\nicefrac{1}{\log
N}$. We see that the groups that known to be amenable are clearly converging to
$3$.  Again we see that Thompson's group behaves quite
differently.}
 \label{fig ev amen}
\end{figure}

Note that no sequence gives a perfectly straight line and so to estimate
$\alpha_\infty$ we fitted the data to the form
\begin{align*}
  \alpha_n &= \alpha_\infty + \nicefrac{\lambda}{(\log N)^\delta}.
\end{align*}
We varied the number of data points by removing small-$N$ points and we also
varied the value of $\delta$. For any fixed number of points we varied $\delta$
to find a value that minimised the $R^2$ statistic. This gives an ``optimal''
value of $\alpha_\infty$ and $\lambda$.

For some groups, we found that these optimal values were quite sensitive to
changes in $\delta$, while other groups were quite robust. To include some
measure of this systematic error we moved $\delta$ through a range
of values so that the $R^2$ statistic was allowed to move to $5\%$ below its
optimal value. These results are summarised in Tables~\ref{tab fit nonamen}
and~\ref{tab fit amen}.

The results for all the groups \emph{except} Thompson's group are as one
might expect --- the amenable groups all give estimates of $\alpha_\infty$ close
to $3$, and the non-amenable groups give $\alpha_\infty < 3$. Hence it would
appear as though this technique is a reasonable test to differentiate amenable
and non-amenable groups. Unfortunately it is not sufficiently sensitive to
determine the amenability of Thompson's group. In particular we find that the
results are too sensitive to variations in $\delta$ and to removal of low-$N$
data points.
A possible reason for this atypical behaviour is the presence of nested wreath products 
which converge very slowly to their asymptotic behaviour.

Because of this, we turn to numerical methods based on random sampling and
approximate enumeration.

\begin{table}[ht!]
\begin{center}
\begin{tabular}{||c|c||c|c|c||}
\hline
\hline
Group&Number&Optimal & $\delta$ range& $\alpha_\infty$
estimate\\
&of points&$R^2$ Value& & \\
\hline
\hline
\multirow{2}{*}{$BS(2,2)$} 
 & 4501 & $0.998$ & $1.74\pm0.05$ & $2.682\pm0.007$ \\
 & 2500 & $0.998$ & $1.85\pm0.13$ & $2.672\pm0.009$ \\
\hline
\multirow{2}{*}{$BS(2,3)$} 
 & 4501 & $0.999$ & $1.36\pm0.04$ & $2.597\pm0.012$ \\
 & 2500 & $0.999$ & $1.57\pm0.07$ & $2.562\pm0.009$ \\
\hline
\multirow{2}{*}{$BS(3,5)$} 
& 4101 & $0.998$ & $1.33\pm0.05$ & $2.29\pm0.01$ \\
& 2000 & $0.998$ & $1.65\pm0.19$ & $2.24\pm0.03$ \\
\hline
\multirow{3}{*}{$F$} 
 & 3947 & $0.998$ & $0.83\pm0.07$ & $2.79\pm0.08$ \\
 & 2000 & $0.998$ & $0.93\pm0.16$ & $2.69\pm0.12$ \\
 & 1700 & $0.998$ & $0.65\pm0.21$ & $2.95\pm0.38$ \\
\hline
\hline
\end{tabular}
\end{center}
\caption{Results of fitting eigenvalue data for non-amenable groups and Thompson's group. The
Baumslag-Solitar groups all give good results, but Thompson's group does not.
There is some upward drift in the estimate of $\alpha_\infty$ as one cuts out
small $N$ data, but at the same time the error in the estimates blows up.}
\label{tab fit nonamen}
\end{table}

\begin{table}[ht!]
\begin{center}
\begin{tabular}{||c|c||c|c|c||}
\hline
\hline
Group&Number&Optimal & $\delta$ range& $\alpha_\infty$
estimate\\
&of points&$R^2$ Value& & \\
\hline
\hline
\multirow{2}{*}{ $BS(1,2)$ }
 & 4501 & $0.99975$ & $1.7316\pm0.0225$ & $3.0158\pm0.0031642$ \\
 & 2500 & $0.99981$ & $1.9472\pm0.0552$ & $2.9975\pm0.0031542$ \\
\hline
\multirow{2}{*}{  $BS(1,3)$ }
& 4501 & $0.99894$ & $1.354\pm0.046$ & $3.0722\pm0.016473$ \\
 & 2500 & $0.99855$ & $1.54\pm0.151$ & $3.0261\pm0.026664$ \\
\hline
\multirow{2}{*}{ $\mathbb{Z}^2$ }
 & 4501 & $0.99613$ & $5.1624\pm0.1134$ & $3.002\pm0.000364$ \\
 & 2500 & $0.99932$ & $10.996\pm0.154$ & $3\pm 1.7248 \times 10^{-6}$ \\
\hline 
\multirow{2}{*}{ $\mathbb{Z}\wr\mathbb{Z}$ }
& 3947 & $0.99925$ & $0.8592\pm0.0436$ & $3.1807\pm0.050903$ \\
& 2000 & $0.99915$ & $1.0237\pm0.1251$ & $3.0476\pm0.082052$ \\
\hline 
\multirow{3}{*}{$F$}
 & 3947 & $0.99796$ & $0.83\pm0.072$ & $2.7866\pm0.083778$ \\
 & 2000 & $0.99869$ & $0.9344\pm0.1548$ & $2.6917\pm0.12318$ \\
 & 1700 & $0.99848$ & $0.6464\pm0.2016$ & $2.9532\pm0.38051$ \\
\hline
\hline
\end{tabular}
\caption{Results of fitting eigenvalue data for amenable groups  and Thompson's group. All the
amenable groups give good results quite close to $3$, though $\mathbb{Z} \wr
\mathbb{Z}$ is not as good as the others. Also note that since balls in
$\mathbb{Z}^2$ grow quadratically with radius rather than exponentially, better
results can be obtained by fitting against $\displaystyle \nicefrac{1}{N^\delta}$ rather than
$\displaystyle \nicefrac{1}{(\log N)^\delta}$.}
\label{tab fit amen}
\end{center}
\end{table}

\subsection{An aside --- cogrowth series}
As a byproduct of our computations we obtained the first few terms of the
cogrowth series for all of these groups. It is well know that the number of
trivial words in $\mathbb{Z}^2$ is given by $\binom{2n}{n}^2$ (see A002894 \cite{Sloane}); the corresponding
generating function is not algebraic and is expressible as a complete elliptic
integral of the first kind. The number of trivial words in $F_2$ is just the
number of returning paths in a quadtree and its generating function is
$3(1+2\sqrt{1-12z^2})^{-1}$ (see A035610 \cite{Sloane}). 

Unfortunately we have been unable to find (using tools such as GFUN
\cite{Salvy1994}) any useful explicit or implicit expressions for the
cogrowth series (or the generating functions) for any of the other groups we
have examined. For completeness we include our data in Table~\ref{tab enum}.
\begin{table}[ht!]
\begin{tabular}{|c||c|c|c|c|c|c|c|}
\hline
n&  $F$ 
& BS(1,2) & BS(1,3) & BS(2,2)& BS(2,3) & BS(3,5) & $\mathbb{Z} \wr
\mathbb{Z}$ \\
\hline
0&1&1&1&1&1&1&1 \\
1&0&0&0&0&0&0&0 \\
2&0&0&0&0&0&0&0 \\
3&0&0&0&0&0&0&0 \\
4&0&0&0&0&0&0&0 \\
5&0&10&0&0&0&0&0 \\
6&0&0&12&12&0&0&0 \\
7&0&20&0&0&14&0&0 \\
8&0&64&40&40&0&0&16 \\
9&0&96&0&0&28&0&0 \\
10&20&338&264&224&60&20&72  \\
11&0&736&0&0&84&0&0 \\
12&64&2052&1604&1236&240&64&272 \\
13&0&5208&0&0&564&0&0 \\
14&336&13336&9748&7252&1090&280&1504 \\
15&0&36330&0&0&2760&0&0 \\
16&1160&92636&61720&41192&6492&1048&8576 \\
17&0&248816&0&0&13496&0&0  \\
18&5896&665196&412072&247272&33728&4660&46080 \\
19&0&1771756&0&0&75768&0&0 \\
20&24652&4776094&2750960&1491136&174760&17964&257160 \\
21&0&12848924&0&0&411234&0&0 \\
22&117628&34765448&18725784&9119452&958364&77508&1475592 \\
\hline
\end{tabular}
\caption{The first few terms of the cogrowth series $C(z)$ for various groups, {\em 
i.e.} the number of freely reduced words equivalent to the identity. The first few
terms of the returns series $R(z)$ can be obtained from the above using
Lemma~\ref{lem:RC}.}
\label{tab enum}
\end{table}

\section{Distribution of geodesic lengths}

In this section we broaden our study from the  growth rate of  trivial words to the
distribution of geodesic lengths of all words by sampling random words. In
previous work of Burillo {\em et al} \cite{Burillo2007}, random words in Thompson's group $F$ were
sampled using simple sampling; words were grown by appending  generators
one-by-one uniformly at random. Those authors observed only very trivial words
and so then sampled uniformly at random from a subset of those words, namely the
set of words with balanced numbers of each generator and their inverses. Again,
very few trivial words were observed. Indeed if Thompson's group is
non-amenable, the probability of observing a trivial word using simple sampling
will decay exponentially quickly.

We will proceed along a similar line but using a more powerful random sampling
method based on flat-histogram ideas used in the FlatPERM algorithm
\cite{Prellberg2004, Prellberg2006}. Each sample word is grown in a similar
manner to simple sampling --- append one generator at a time chosen uniformly at
random. The weight of a word of $n$ symbols is simply $1$, so that the
total weight of all possible words at any given length is just $4^n$. As the
word grows we keep track of its geodesic length. We now deviate from simple
sampling by ``pruning'' and ``enriching'' the words.

Consider a word of length $n$, geodesic length $\ell$ and weight $W$. If we
have ``too many'' samples of such words, then with probability $\nicefrac{1}{2}$
prune the current sample or otherwise continue to grow the current sample
but with weight $2W$. Similarly if we have ``too few'' samples of the current
length and geodesic length, then enrich by making $2$ copies of the current
word and then growing a sample from both each with weight $\nicefrac{W}{2}$. Of
course, one is free to play around with the precise meaning of ``too few'' or
``too many''. We refer the reader to \cite{Prellberg2004, Prellberg2006} for
more details on the implementation of this algorithm. The mean weight
(multiplied by $4^n$) of all samples of length $n$ and geodesic length $\ell$,
$c_{n,\ell}$,  is then an estimate of the number of such words.

In order to run the above algorithm we need to be able to compute the geodesic
length of the element generated by a given random word. Computing geodesic
lengths from a normal form is, in general, a very difficult problem and remains
stubbornly unsolved for many interesting groups, such as $BS(2,3)$. Because of
this we restrict our studies to Thompson's group and a number of different
wreath products.
\begin{itemize}
 \item Thompson's group --- a method for computing the geodesic length of an
element from its tree-pair representation was first given by Fordham
\cite{Fordham2003}, though we found it easier to implement the method of Belk
and Brown \cite{Belk2005}.
 \item Wreath products --- we use the results of \cite{Cleary2005} to find the
geodesic lengths in $\mathbb{Z} \wr \mathbb{Z}$,
$\mathbb{Z}\wr(\mathbb{Z}\wr\mathbb{Z})$ and $\mathbb{Z} \wr F_2$.
\end{itemize}
We note that the geodesic problem for Baumslag-Solitar groups has
recently been solved in the cases $BS(1,n)$ \cite{Elder2009} and $BS(n,kn)$
\cite{Diekert2011}, but we have not implemented these approaches. 

\subsection{Distributions}
We used the random sampling algorithm described above to estimate the
distribution of geodesic lengths in Thompson's group $F$, as well as
$\mathbb{Z} \wr \mathbb{Z}$, $\mathbb{Z} \wr F_2$ and $\mathbb{Z} \wr
(\mathbb{Z} \wr \mathbb{Z})$. Each run took approximately 1 day on a modest
desktop computer. To visualise the results, we started by normalising the data
by dividing by the total number of words ({\em i.e.} $4^n$ or $6^n$). The resulting
peak-heights still decay with length, and we found that multiplying by
$\sqrt{n}$ compensated for this. The normalised distributions are plotted in
Figures~\ref{fig dist tgrpf},~\ref{fig dist zwrz} and~\ref{fig dist zwrothers}.

In each case we see similar behaviour. At short word lengths ({\em i.e.} small $n$) the
distribution of geodesic lengths is quite wide, but settles to what appears to
be a bell-shaped distribution at moderate lengths. This suggests that the
geodesic length has an approximately Gaussian distribution about the mean
length and that the tails of the distribution are exponentially suppressed. 
This also explains why the normalising factor of $\sqrt{n}$ works well.

If this is indeed the case, then we expect that trivial words, having geodesic
length zero, will be exponentially fewer than $4^n$ --- implying that
Thompson's group is non-amenable. Unfortunately things cannot be so simple,
because the same reasoning would imply that $\mathbb{Z} \wr \mathbb{Z}$ is
non-amenable. 

One obvious difference between the graphs is the movement of the peak of the
distribution, that is the rate of growth of the mean geodesic length. It is
clear that the mean geodesic length of $\mathbb{Z} \wr F_2$ grows linearly, and
so the group has a nontrivial rate of escape --- exactly as one would expect
of a non-amenable group. Similarly we see that the mean geodesic lengths
of the other wreath products grow sublinearly, so their rates of escape are
zero. When we examine the movement of the peak of Thompson's group's
distribution, things are less clear; the mean geodesic length appears to be very
nearly linear.

\begin{figure}[ht!]
 \begin{center}
  \includegraphics[height=8cm]{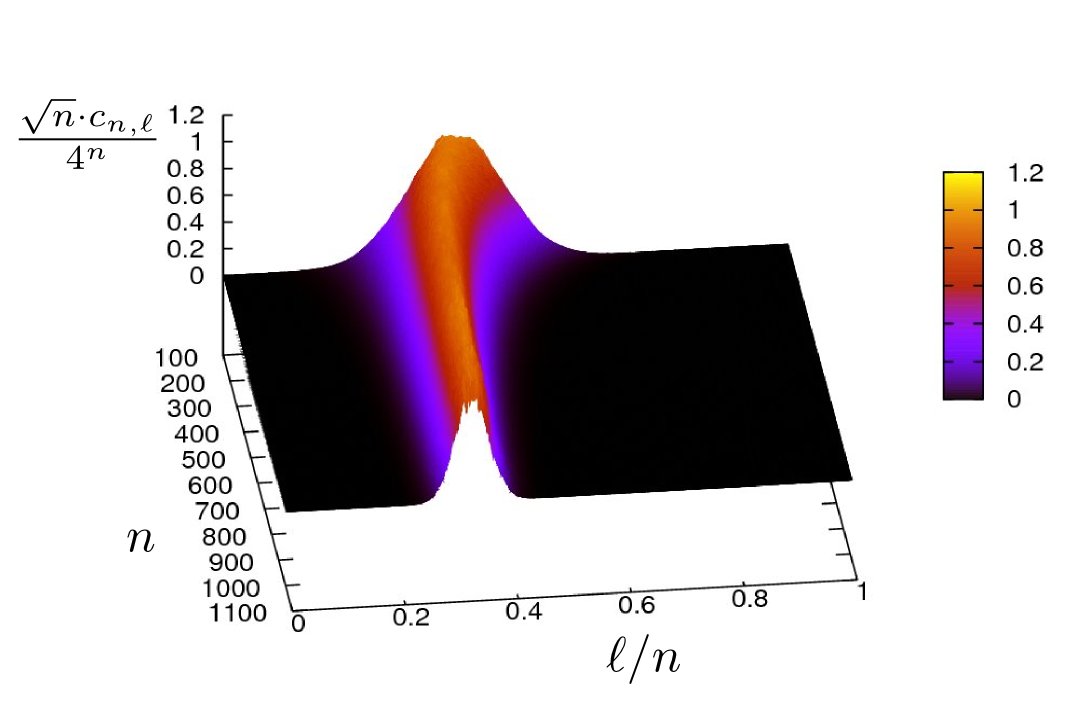}
 \end{center}
  \caption{A plot of the normalised distribution of the number of words 
$c_{n,\ell}$ of length $n$ and geodesic length $\ell$ in Thompson's group $F$.
Notice that the peak position is quite stable, indicating that the mean
geodesic length grows roughly linearly with word length.}
 \label{fig dist tgrpf}
\end{figure}

\begin{figure}[ht!]
 \begin{center}
  \includegraphics[height=8cm]{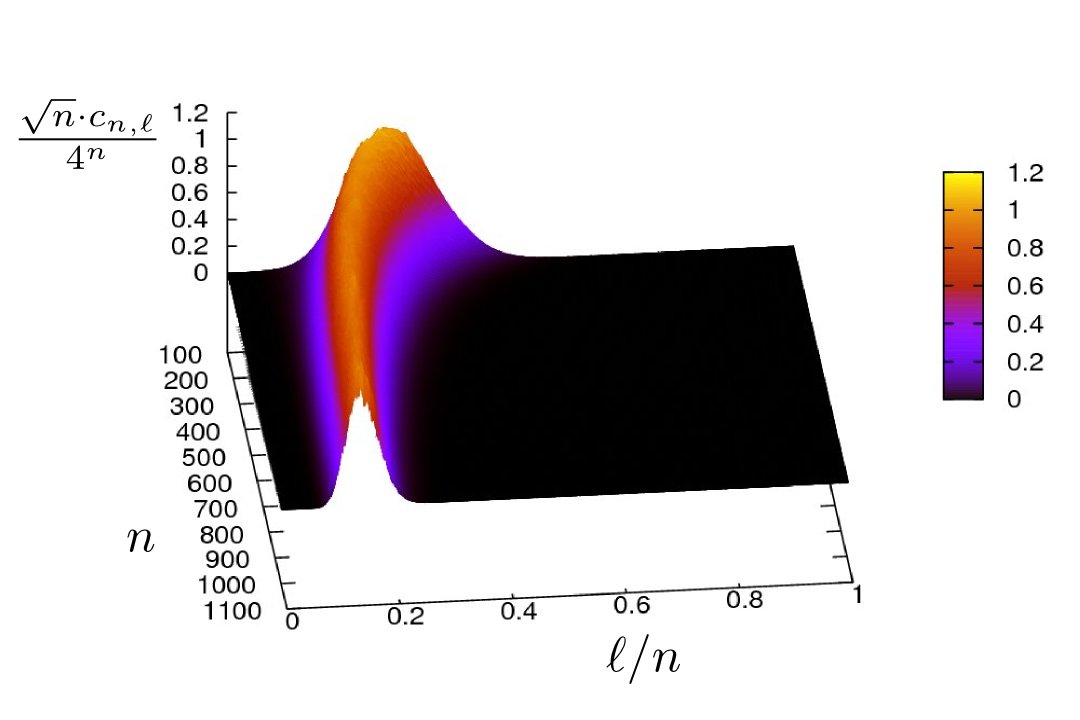}
 \end{center}
  \caption{A plot of the normalised distribution of the number of words 
$c_{n,\ell}$ of length $n$ and geodesic length $\ell$ in $\mathbb{Z} \wr
\mathbb{Z}$. Observe that the peak position is clearly moving towards the left
of the plot suggesting that the mean geodesic length grows sublinearly.}
 \label{fig dist zwrz}
\end{figure}

\begin{figure}[ht!]
 \begin{center}
  \includegraphics[width=7cm]{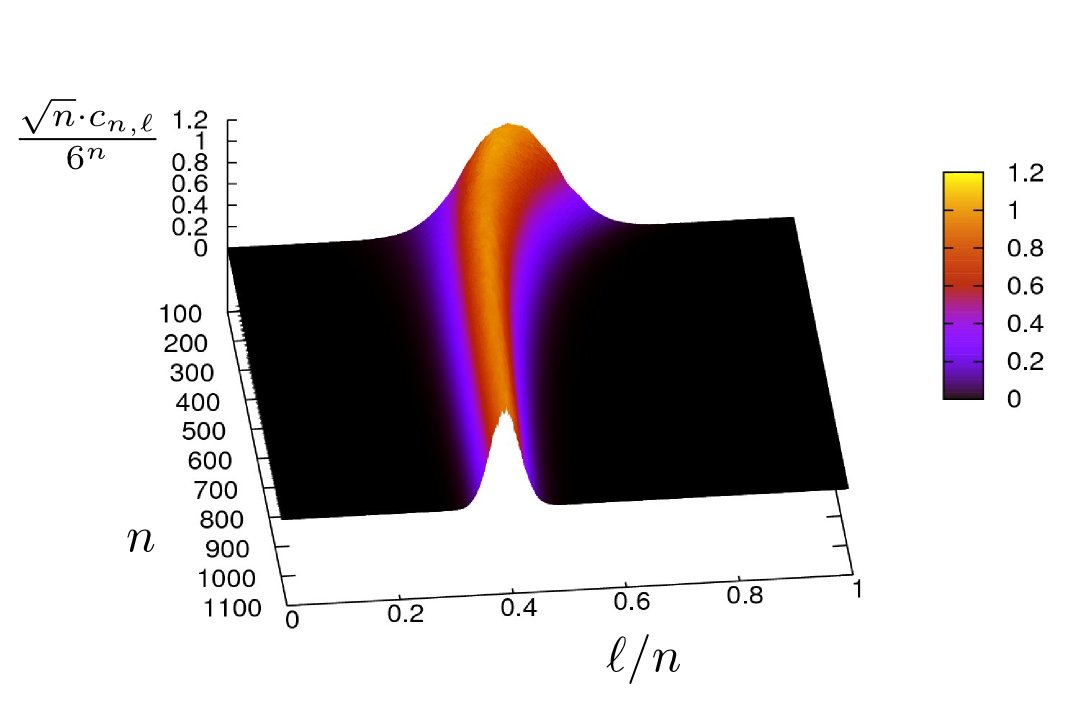}
  \includegraphics[width=7cm]{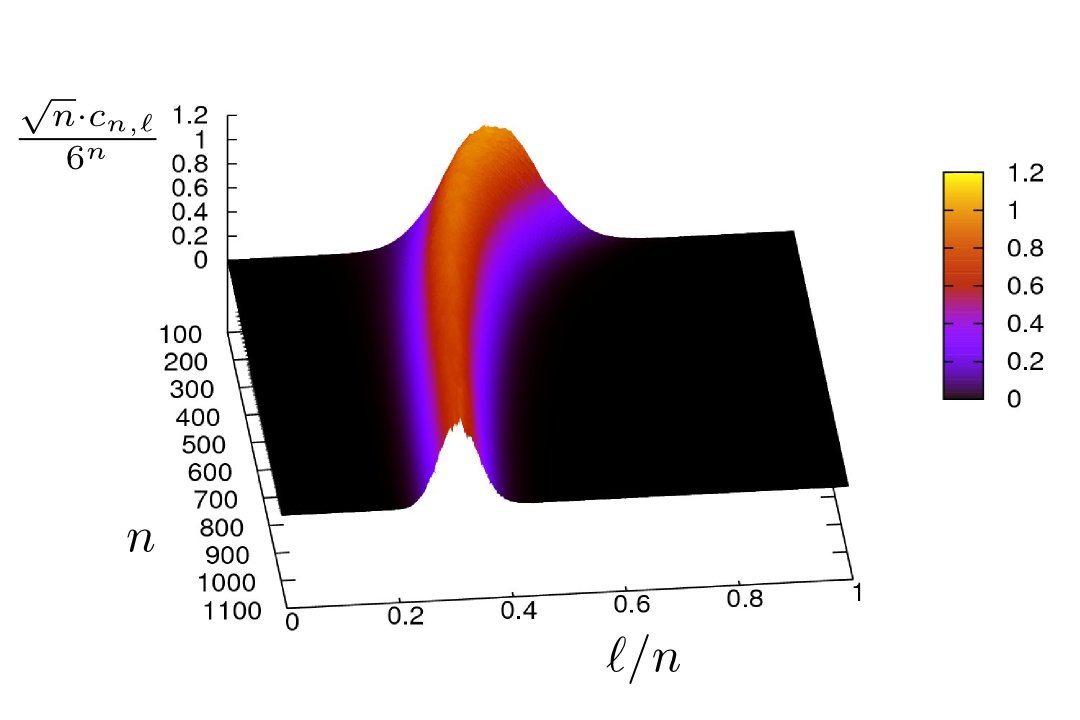}
 \end{center}
  \caption{Plot of the normalised distribution of the number of words 
$c_{n,\ell}$ of length $n$ and geodesic length $\ell$ in $\mathbb{Z} \wr F_2$
(left) and $\mathbb{Z} \wr (\mathbb{Z} \wr \mathbb{Z})$ (right). Observe that the
peak is quite stable in the left-hand plot indicating the mean geodesic length
is linear, while the right-hand plot the peak shows clear a left drift
indicating that the geodesics grow sublinearly.}
 \label{fig dist zwrothers}
\end{figure}

Estimating the mean geodesic length for Thompson's group was substantially
easier. We constructed $2^{12}$ random words of length $2^{16}$. As each word
was constructed generator-by-generator, the geodesic length was computed and
added to our statistics. So while there is correlation between the geodesic
lengths at different word lengths \emph{within} a given sample, there is no
correlation \emph{between} samples. This took approximately 3 days on a
modest desktop computer. Our data is plotted in Figure~6. 

\begin{figure}[ht!]
 \begin{center}
 \includegraphics[height=7cm]{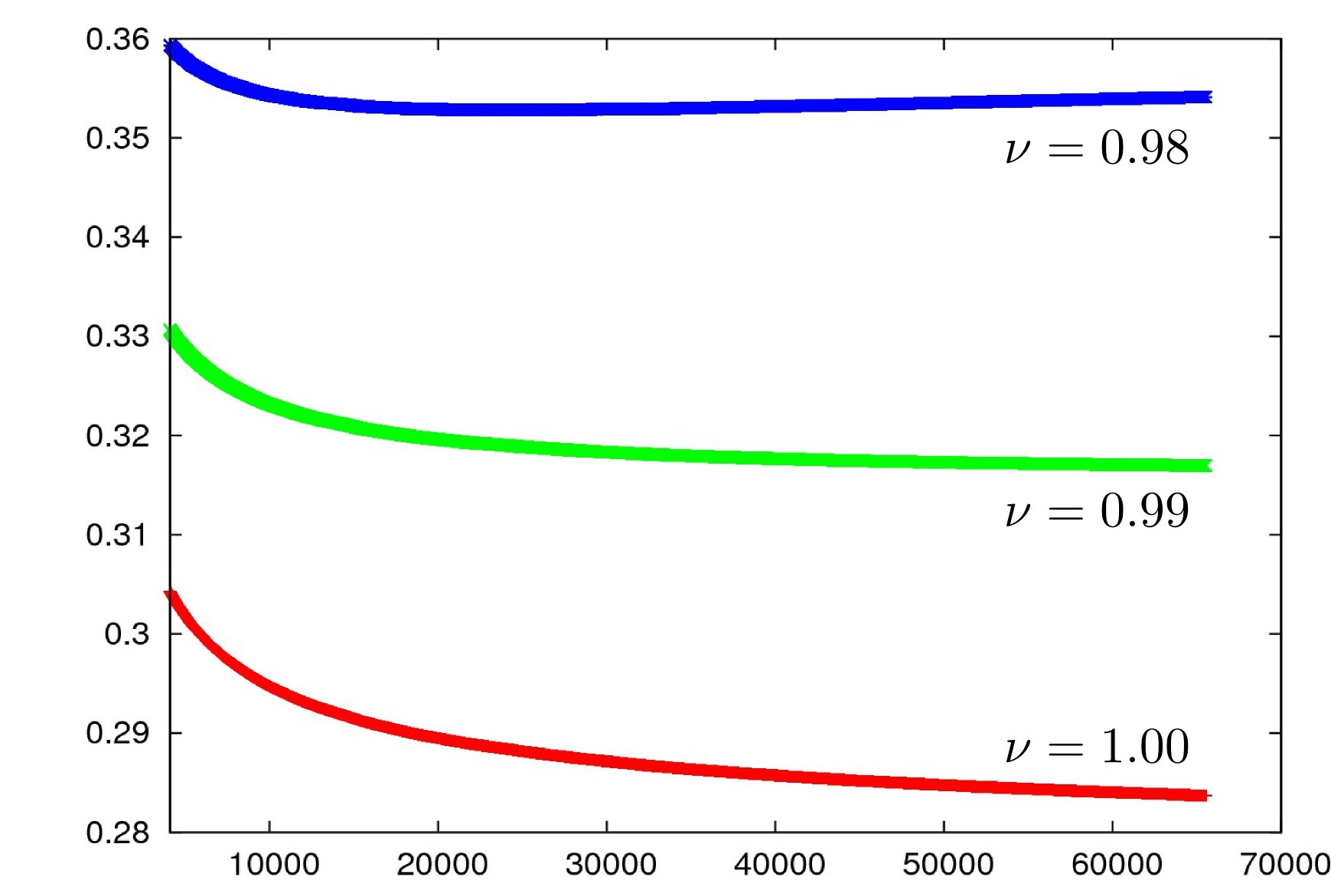}
 \end{center}
 \label{fig-mean-geods}
  \caption{Plot of the mean geodesic length divided by $n^\nu$; for $\nu = 0.98,
0.99$ and $1$. This data strongly suggests that Thompson's group has a
nontrivial rate of escape. Note that the statistical error was smaller than
the symbols used.}
\end{figure}

We assume that the mean geodesic length grows as $n^\nu$. Linear regression on
a log-log plot estimates $\nu \approx 0.98$. Further, if we fit a moving
``window'', we find that the local estimates of $\nu$ increase as the
positioning of the window increases. This strongly suggests that the mean
geodesic length grows linearly. 

To test linearity further, we generated a small number words of length
$2^{20} = 1048576$. It took approximately 1 hour to generate each word and
compute the corresponding geodesic length, so this was too slow to generate
meaningful statistics. In each case we observed that the ratio
$\nicefrac{\ell}{n}$ appeared to converge to approximately $0.28$.  Of course,
this does not preclude more exotic sublinear behaviour such as $n^\nu (\log
n)^\theta$. Such logarithmic corrections are extremely difficult to detect or
rule out.

We now estimate the rate of escape by assuming linear growth with a polynomial
subdominant correction term
\begin{align}
  \langle \ell \rangle_n = A n + b n^\delta.
\end{align}
Our estimates were quite sensitive to changes in $\delta$:
\begin{align}
\begin{array}{|c|c|c|c|c|c|c|c|}
\hline
 \delta & 0 & \nicefrac{1}{4} &\nicefrac{1}{3}  & \nicefrac{1}{2} &
\nicefrac{2}{3} & \nicefrac{3}{4} \\
\hline
  A & 0.281 & 0.279 & 0.279 & 0.276 & 0.272 & 0.267 \\
  b & 176 & 17 & 8.0 & 1.8 & 0.47 & 0.25\\
  \hline
\end{array}.
\end{align}
Hence we conclude that the rate of escape is approximately $0.27$ with an error of $\pm 0.01$.

We would like to conclude that this positive rate of escape implies that
Thompson's group is non-amenable, however there are examples of amenable groups
with nontrivial rate of escape. The  group $\mathbb{Z}^3\wr \mathbb{Z}_2$ is amenable but has positive rate of escape \cite{Revelle}. Unfortunately,
computing geodesics in this group is equivalent to solving the traveling
salesman problem on $\mathbb{Z}^3$ \cite{Parry1992} and so beyond these techniques.

\section{Conclusions}

We have computed exact lower bounds on the cogrowth of several groups including
Thompson's group $F$. In particular, the cogrowth ($\alpha$) of Thompson's group must be
greater than $2.17329$. By extrapolating the sequences of lower bounds we see
that the bounds for the amenable groups clearly converge to $3$, while those of
the non-amenable groups converge to numbers strictly less than $3$. Thompson's
group appears to behave quite differently from the other groups we examined. Our
extrapolations do not give clear results, though perhaps they point towards
non-amenability.

To further probe this group we used flat histogram methods to estimate the
distribution of geodesic lengths in random words. The data suggests that
geodesic lengths have an approximately Gaussian distribution about their mean
length. Similar Gaussian distributions were observed for other groups, both
amenable and non-amenable. 

The mean geodesic length of the amenable groups studied grow sublinearly, while
those of $\mathbb{Z} \wr F_2$ and Thompson's group are observed to grow
linearly. Using simple sampling we estimate that the mean geodesic length of
Thompson's group does indeed grow linearly  and that the rate of escape
is $0.27\pm 0.01$. 

\section{Acknowledgments}
We thank the anonymous reviewers for their helpful comments, and WestGrid for computer support.

\end{document}